\documentclass[10pt,a4paper]{article}
\usepackage[margin=1.1cm]{geometry}
\usepackage{amsmath,amsthm,amsfonts,amssymb,amscd,cite,graphicx}

\usepackage{titlesec}
\titleformat{\section}
{\normalfont\fontsize{12}{15}\bfseries}{\thesection}{1em.}{}

\usepackage[labelfont=bf]{caption}
\counterwithin{figure}{section}
\counterwithin{table}{section}

\usepackage{float}

\usepackage{tikz}
\usetikzlibrary{shapes.geometric, positioning, calc}

\newcommand{\Z}{{\mathbb Z}}

\newcommand{\ICG}{\mathrm{ICG}}

\newtheorem{lemma}{Lemma}[section]

\newtheorem{theorem}{Theorem}[section]

\allowdisplaybreaks[4]

\let\oldbibliography\thebibliography
\renewcommand{\thebibliography}[1]{%
  \oldbibliography{#1}%
  \setlength{\itemsep}{-2pt}%
}

\begin{document}

\baselineskip=0.20in

\makebox[\textwidth]{%
\hglue-15pt
\begin{minipage}{0.6cm}	
\vskip9pt
\end{minipage} \vspace{-\parskip}
\begin{minipage}[t]{6cm}
\end{minipage}
\hfill
\begin{minipage}[t]{6.5cm}
\end{minipage}}
\vskip36pt

\noindent
{\large \bf On uniquely colorable Cayley graphs}\\

\vspace{0.5cm}

\noindent
Milan Ba\v{s}i\'{c}$^{1,}\footnote{Corresponding author (milan.basic@pmf.edu.rs, basic\_milan@yahoo.com)}$\\

\noindent
\footnotesize $^1${\it Department of Computer Science, Faculty of Sciences and Mathematics, University of Ni\v{s}, Vi\v{s}egradska 33, 18000 Ni\v{s}, Serbia} \\

\noindent
 (\footnotesize Received: Day Month 202X. Received in revised form: Day Month 202X. Accepted: Day Month 202X. Published online: Day Month 202X.)\\

\setcounter{page}{1} \thispagestyle{empty}

\baselineskip=0.20in

\normalsize

 \begin{abstract}
 \noindent
We resolve two open problems regarding uniquely colorable Cayley graphs posed by Klotz and Sander (2017). First, we construct an infinite family of uniquely 3-colorable integral circulant graphs with a clique number of 2. This provides a negative answer to Problem 3.6, which asks whether every uniquely colorable circulant graph satisfies $\chi(G) = \omega(G)$. Because verifying unique colorability inherently relies on the exact independence number, we demonstrate that traditional spectral bounds fail to tightly capture this parameter, necessitating a rigorous combinatorial proof based on exact structural isomorphisms. Second, we establish a general algebraic construction proving the existence of uniquely colorable Cayley graphs over nonabelian groups whose color classes are left cosets of strictly distinct subgroups. By utilizing right-coset partitions of non-normal subgroups, this result provides a definitive affirmative answer to Problem 2.4.
\\[2mm]
\end{abstract}

\baselineskip=0.20in

\section{Introduction}


%
%
%

%

Integral circulant graphs and their broader algebraic generalizations, Cayley graphs, form a prominent class of highly symmetric interconnection networks. Their topological regularity makes them exceptionally viable models for quantum spin networks, facilitating phenomena such as perfect state transfer, uniform mixing, and quantum walks. This applicability spans from classical continuous-time models \cite{ahmadi03, Ba14, Godsil17} to recent discrete-time Grover walks on both undirected and directed Cayley topologies over diverse algebraic structures \cite{Bhakta-2025, Bhakta-2026, Kubota-2025}. Beyond these applied physics models, there is a continuous mathematical effort to resolve the fundamental combinatorial and spectral properties of integral circulant graphs. Recent literature has seen significant structural advancements by various authors, briefly encompassing investigations into graph isospectrality \cite{Moenius-2023}, geometric kernels \cite{Sander2021}, families of graphs with few distinct eigenvalues \cite{Basic-LAA-2025, SanderSander2023}, as well as the exact determination of maximal diameters and spectral bounds \cite{BaIlSt23, Basic-2024, Basic-2025}.

A proper $k$-coloring of a graph $G$ is a partition of its vertex set into $k$ independent sets, called color classes. A graph is \textit{uniquely $k$-colorable} if every optimal coloring induces the exact same partition of vertices. Identifying uniquely colorable graphs and determining their exact chromatic number is a notoriously difficult task. Historically, the study of these parameters in Cayley networks was initiated by Klotz and Sander \cite{klotz07}. While the clique number was completely characterized for integral circulant graphs with few divisors \cite{Basic-2009}, determining the exact chromatic number has proven far more elusive. For instance, in \cite{Basic-2010}, the chromatic number was exactly determined for several classes of integral circulant graphs with two divisors; yet, for one specific class, only an upper bound was established alongside an open conjecture that equality holds. This persisting challenge perfectly illustrates the inherent mathematical complexity of exact vertex coloring in such highly symmetric networks.

The foundational result that unitary Cayley graphs are uniquely colorable was established by Ili\'{c} and Ba\v{s}i\'{c} \cite{Basic-2010}. Motivated by this discovery, Klotz and Sander's subsequent 2017 study \cite{Klotz-2017} investigated the strict algebraic and structural constraints imposed by unique colorability across broader graph classes, leading them to pose several open conjectures. This paper systematically resolves two of their most prominent open questions, one operating in the abelian (circulant) domain, and the other in the nonabelian domain.

The first question, posed as Problem 3.6 in \cite{Klotz-2017}, asks whether every uniquely colorable circulant graph must satisfy $\chi(G) = \omega(G)$. The intuition behind this conjecture relies on the fact that highly symmetric uniquely colorable graphs often attain the tightest possible bounds for their independence number, naturally forcing the existence of large cliques. In Section 3, we provide a negative answer to this problem by constructing an infinite family of uniquely $3$-colorable integral circulant graphs with a clique number of $2$. Because verifying unique colorability inherently relies on determining the exact independence number, we demonstrate that standard spectral techniques, such as the Hoffman ratio bound, fail to tightly capture this parameter. This explicitly necessitates a novel combinatorial approach based on structural isomorphisms to confirm the strict inequality between the chromatic and clique numbers.

The second question, posed as Problem 2.4 in \cite{Klotz-2017}, explores the algebraic generation of color classes. The authors proved that the color classes of any uniquely colorable Cayley graph are left cosets of conjugate subgroups. They questioned whether there exists a Cayley graph over a nonabelian group where these generating subgroups are strictly distinct. In Section 4, we provide a definitive affirmative answer to this problem. We introduce a general structural theorem utilizing right-coset partitions of non-normal subgroups to bypass the conjugacy limitations, validating the conjecture for a broad class of nonabelian groups.

\section{Preliminaries}
\label{sec:1}

A {\it circulant graph} $G(n;S)$ is a graph on vertices $\Z_n=\{0,1,\ldots,n-1\}$ where vertices $i$ and $j$ are adjacent if and only if $i-j \equiv s \pmod n$ for some $s \in S$. The set $S$ is called the {\it symbol} of $G(n;S)$. We consider only undirected graphs without loops, assuming $S=n-S=\{n-s\ |\ s\in S\}$ and $0\not\in S$. The degree of $G(n;S)$ is exactly $|S|$.

The eigenvalues and corresponding eigenvectors of $G(n;S)$ are given by:
\begin{equation*} \label{eq:eigenvalues unwigted} 
\lambda_j=\sum_{s \in S} \omega^{js}_n, \quad v_j=[1 \ \omega_n^j \ \omega_n^{2j} \cdots \omega_n^{(n-1)j}]^T, \quad 0\leq j\leq n-1,
\end{equation*}
where $\omega_n=e^{i\frac{2\pi}n}$ is the $n$-th root of unity \cite{Davis70}.

\smallskip

Circulant graphs belong to the broader class of Cayley graphs. Let $\Gamma$ be a multiplicative group with identity $e$. For a subset $S \subset \Gamma$ where $e\not\in S$ and $S^{-1} = \{s^{-1}\ |\ s\in S\}=S$, the Cayley graph $X = Cay(\Gamma,S)$ is the undirected graph with vertex set $V(X)=\Gamma$ and edge set $E(X) = \{\{a,b\}\ |\ ab^{-1}\in S\}$. A graph is circulant if it is a Cayley graph on a cyclic group, meaning its adjacency matrix is cyclic.

A graph is {\it integral} if all its eigenvalues are integers. A circulant graph $G(n;S)$ is integral if and only if
$$
S=\bigcup_{d \in D} G_n(d),
$$
for some set of divisors $D \subseteq D_n$ \cite {wasin}. Here $G_n(d)=\{ k \ : \ \gcd(k,n)=d, \ 1\leq k \leq n-1 \}$, and $D_n$ is the set of all divisors of $n$, strictly less than $n$.

Therefore, an {\it integral circulant graph} (hereafter $\ICG$) is entirely defined by its order $n$ and the set of divisors $D \subseteq D_n$, and is denoted by $\ICG_n(D)$. A prominent example is the unitary Cayley graph, defined as $\ICG_n(\{1\})$. The degree of $\ICG_n(D)$ is given by $\deg \ICG_n(D)=\sum_{d \in D}\varphi(\frac{n}{d})$, where $\varphi(n)$ is the Euler totient function \cite{HardyWright}. If $D=\{d_1,\ldots,d_k\}$, it is easy to see that $\ICG_n(D)$ is connected if and only if $\gcd(d_1,\ldots,d_k)=1$ \cite{Hwang03}.

Let $c(j,n)$ denote the {\it Ramanujan function} \cite[p.~55]{HardyWright}, defined as:
\begin{equation}
c(j,n)=\mu(t_{n,j})\frac{\varphi(n)}{\varphi(t_{n,j})}, \quad t_{n,j}=\frac n{\gcd(n,j)}, \quad 0\leq j\leq n-1,
\label{ramanujan}
\end{equation} 
where $\mu$ is the M\"obius function:
\begin{eqnarray*}
\mu(n)&=&\left\{
\begin{array}{rl}
1, &  \mbox{if}\ n=1  \\
0, & \mbox{if $n$ is not square-free} \\
(-1)^k, & \mbox{if $n$ is a product of $k$ distinct primes}.
\end{array} \right.
\end{eqnarray*}
Since $\gcd(n,j)=\gcd(n,n-j)$, it directly follows from (\ref{ramanujan}) that $t_{n,j} = t_{n,n-j}$, yielding the symmetry property $c(j,n)=c(n-j,n)$ for $0\leq j\leq n$. The eigenvalues of $\ICG_n(D)$ can be expressed entirely in terms of the Ramanujan function \cite{wasin} as:
\begin{equation}
\lambda_j(n,D)=\sum_{d\in D} c(j,\frac{n}{d}),
\label{ldef} 
\end{equation}
where $0\leq j\leq n-1$. Whenever clear from the context, we will abbreviate $\lambda_j(n,D)$ to $\lambda_j$.


\section{The uniquely colorable graph family $G_p$}

While the broader motivation of this paper encompasses multiple open conjectures regarding uniquely colorable graphs, this specific section is dedicated to resolving Problem 3.6 posed by Klotz and Sander in 2017 \cite{Klotz-2017}. This problem asks whether every uniquely colorable circulant Cayley graph must satisfy $\chi(G) = \omega(G)$. Determining unique colorability and the exact chromatic number of a graph is notoriously difficult, as there is no universal algorithmic test. The foundational approach relies on finding the exact independence number $\alpha(G)$. If the independence number is exactly $\alpha(G) = \vert{}V(G)\vert{} / \chi(G)$, every color class is forced to be a maximum independent set. When a graph's structural topology strictly restricts such maximum independent sets to a single unique partition, unique colorability is rigidly guaranteed. Thus, finding a tight bound for $\alpha(G)$ is the critical first step.

A standard algebraic tool to upper bound $\alpha(G)$ is the Hoffman ratio bound. Using the graph's spread $S(G) = \lambda_{max} - \lambda_{min}$ and noting $\lambda_{min} < 0$, this bound can be elegantly rewritten as a trade-off between the absolute least eigenvalue and the spread:
\begin{equation}\label{eq:hoffman-tradeoff}
    \alpha(G) \le \frac{|V(G)|}{1 - \frac{\lambda_{max}}{\lambda_{min}}} = |V(G)| \frac{|\lambda_{min}|}{S(G)}.
\end{equation}

To minimize this bound and force unique colorability, a natural strategy is to maximize both the spread and $|\lambda_{min}|$. In \cite{Basic-2025}, it was shown that among all connected integral circulant graphs of order $n$, the graph $\ICG_n(\overline{D}_{p_1})$ achieves the absolute optimal trade-off, yielding a perfectly tight bound of $\alpha \le n/p_1$. However, such spectrally extremal graphs contain cliques of size $p_1$, satisfying $\chi(G) = \omega(G) = p_1$. Consequently, they naturally support, rather than disprove, the Klotz-Sander conjecture.

On the opposite end of the spectral extremes, graphs minimizing the spread \cite{Basic-2024} experience a notable failure of the Hoffman bound. By Theorem 5 in \cite{Basic-2024}, for an order $n = 12p = 2^2 \cdot 3^1 \cdot p^1$ decomposing into $s=3$ prime-power components, the minimal spread is achieved by the divisor set $D = \{12p/4, 12p/3, 12p/p\} = \{3p, 4p, 12\}$. 
According to the theorem, the minimal spread evaluates explicitly as the sum of these prime-power components:
\begin{equation*}
    S(G_{min}) = \sum_{i=1}^s p_i^{\alpha_i} = 4 + 3 + p = p + 7.
\end{equation*}

For this extremal graph, the maximum eigenvalue evaluates to the degree $\lambda_{max} = \sum_{i=1}^s \phi(p_i^{\alpha_i}) = \phi(4) + \phi(3) + \phi(p) = 2 + 2 + (p-1) = p + 3$. Using the spread definition $S(G_{min}) = \lambda_{max} + |\lambda_{min}|$, we directly extract the absolute minimum eigenvalue as $|\lambda_{min}| = (p+7) - (p+3) = 4$. Substituting these parameters into the right-hand side of (\ref{eq:hoffman-tradeoff}) cleanly demonstrates the reduction:
\begin{equation*}
    \alpha(G_{min}) \le |V(G_{min})| \frac{|\lambda_{min}|}{S(G_{min})} = 12p \frac{4}{p+7} = \frac{48p}{p+7}.
\end{equation*}

For $p=5$ ($n=60$), the formula evaluates to an exact integer, $\alpha \le 20$. However, the divisor $12$ natively generates exactly $12$ disjoint cliques of size $5$, providing a vertex clique cover that strictly restricts the true independence number to $\alpha(G_{min}) \le 12$. Furthermore, since the minimal divisor is $12$, the consecutive vertices $\{0, 1, \dots, 11\}$ form an independent set of size $12$, proving exactly $\alpha(G_{min}) = 12$. Consequently, $\omega(G_{min}) \le 60/12 = 5$, and since a clique of size $5$ exists, $\omega(G_{min}) = 5$. With an explicit coloring---partitioning the vertices into five sequential blocks $V_i = \{12i, \dots, 12i+11\}$, where no intra-block distance can reach the minimum divisor $12$---the chromatic number is verified as $\chi(G_{min}) = 5$. This reveals a fundamental topological limitation: the minimal spread graph satisfies $\chi(G) = \omega(G) = 5$, failing to resolve the conjecture while the Hoffman bound completely fails to capture its true independence number.

To construct a verifiable counterexample where $\chi(G) > \omega(G)$, we must deliberately sacrifice the minimal spread to gain structural clarity. Retaining the order $n=12p$, we introduce the denser, yet highly structured family $G_p = \ICG_{12p}(2, p)$. 

To test whether spectral methods can tightly bound $\alpha(G_p)$, we must first determine its extreme eigenvalues. The degree (maximum eigenvalue) is given by the sum of the Euler totient functions of its co-divisors:
$$ \lambda_{max} = \varphi(\frac{12p}{2}) + \varphi(\frac{12p}{p}) = \varphi(6p) + \varphi(12) = 2(p-1) + 4 = 2p + 2. $$
The minimum eigenvalue requires a precise evaluation of the Ramanujan sums, as established below.

\begin{lemma}
\label{lem:min_eigenvalue}
For any prime $p \ge 5$, the minimal eigenvalue of the integral circulant graph $G_p = \ICG_{12p}(2, p)$ is $\lambda_{min} = 2 - 2p$, achieved at the index $j = 3p$.
\end{lemma}
\begin{proof}
Let $n = 12p$ and $D = \{2, p\}$. By (\ref{ldef}), the eigenvalues are:
$$ \lambda_j = c(j, 6p) + c(j, 12). $$
We determine the possible values for each Ramanujan term $c(j, x) = \mu(t_{x,j}) \frac{\varphi(x)}{\varphi(t_{x,j})}$, where $t_{x,j} = \frac{x}{\gcd(x,j)}$.

For $x = 6p$, since $p \ge 5$, we have $\varphi(6p) = 2(p-1)$. The strictly negative values occur when $t_{6p,j} \in \{2, 3, p, 6p\}$ (where $\mu = -1$). Evaluating $-\frac{2(p-1)}{\varphi(t_{6p,j})}$ for these divisors yields $2-2p$, $1-p$, $-2$, and $-1$, respectively. Thus, $c(j, 6p) \ge 2-2p$.

For $x = 12$, we have $\varphi(12) = 4$. The non-zero values occur for square-free $t_{12,j} \in \{1, 2, 3, 6\}$. Evaluating $\mu(t_{12,j})\frac{4}{\varphi(t_{12,j})}$ yields $c(j, 12) \in \{4, -4, -2, 2\}$. Thus, $c(j, 12) \ge -4$.

For $j = 3p$, we have $\gcd(6p, 3p) = 3p \implies t_{6p, 3p} = 2$, yielding $c(3p, 6p) = 2-2p$. Simultaneously, $\gcd(12, 3p) = 3 \implies t_{12, 3p} = 4$. Since $4$ is not square-free, $\mu(4) = 0$, giving $c(3p, 12) = 0$. Hence, $\lambda_{3p} = 2-2p$.

To prove this is the strict global minimum, suppose there exists an index $j$ such that $\lambda_j < 2-2p$. Since $c(j, 6p)$ can take values $\{2-2p, 1-p, -2, -1\}$, and $c(j, 12) \ge -4$, the only theoretical way to obtain $\lambda_j < 2-2p$ (given $p \ge 5$) is if the sub-optimal term $c(j, 6p) = 1-p$ combines with a negative $c(j, 12)$, specifically $-4$ or $-2$.

If $c(j, 12) = -4$, then $t_{12,j} = 2$, meaning $\gcd(12, j) = 6$, which forces $j$ to be a multiple of $3$. However, $c(j, 6p) = 1-p$ requires $t_{6p,j} = 3$, meaning $\gcd(6p, j) = 2p$, which explicitly forces $j$ not to be a multiple of $3$. This is a contradiction.

If $c(j, 12) = -2$, the total eigenvalue would be $\lambda_j = (1-p) - 2 = -1-p$. But for all $p \ge 5$, we have $-1-p > 2-2p$, which contradicts the assumption that $\lambda_j < 2-2p$. 

Therefore, no index $j$ can yield a sum smaller than $2-2p$, confirming $\lambda_{min} = 2-2p$.
\end{proof}

With the spread of $G_p$ now established as $S(G_p) = (2p+2) - (2-2p) = 4p$, we evaluate the Hoffman bound:
$$ \alpha(G_p) \le 12p \frac{2p-2}{4p} = 6(p-1). $$
Since our ultimate goal is to prove that the exact independence number is $4p$, this spectral bound is notably loose (e.g., yielding $\alpha \le 24$ instead of the true value $20$ for $p=5$). The failure of the Hoffman bound to tightly capture the true independence number emphasizes that spectral methods alone cannot verify the unique colorability of $G_p$. This explicitly necessitates the rigorous combinatorial approach developed in the following subsection.

\subsection{Structural properties and exact independence number}

Recall the graph family $G_p = \ICG_{12p}(2, p)$ for a prime $p \ge 5$. 
For $i \in \{0,1,2\}$, define the classes $C_i = \{x \in \mathbb{Z}_{12p} \mid x \equiv i \pmod 3\}$. Since $12p$ is divisible by $3$, the sets $C_0, C_1, C_2$ partition $V(G_p)$ into three sets of equal size $4p$.
For any $x, y \in C_i$, the difference $x-y$ is a multiple of $3$. Since $p \ge 5$, neither $2$ nor $p$ is divisible by $3$. Thus, $\gcd(x-y, 12p) \notin \{2, p\}$, meaning there are no edges within any $C_i$. The classes form a proper $3$-coloring, implying $\chi(G_p) \le 3$. By Theorem~5 in \cite{Saxena-2007}, an integral circulant graph is bipartite if and only if its order is even and every divisor in its defining set is odd. Since the order $12p$ is even but $2 \in D$, $G_p$ is not bipartite, yielding exactly $\chi(G_p) = 3$.

\begin{lemma}\label{lem:structure-Gp}
For every distinct $i, j \in \{0,1,2\}$, the induced subgraph $G_p[C_i \cup C_j]$ is isomorphic to $K_2 \otimes (C_4 \square K_p)$.
\end{lemma}

\begin{proof}
By the circulant symmetry, it suffices to analyze $H_r = G_p[C_0 \cup C_r]$ for $r \in \{1,2\}$. 
Since there are no edges within $C_0$ or $C_r$, $H_r$ is bipartite. Let $a = 3l \in C_0$ and $b = 3k + r \in C_r$. 
The difference is $b - a = 3(k - l) + r$. Since $r \in \{1, 2\}$, $b - a$ is not divisible by $3$. 
Because $\gcd(b - a, 3) = 1$, we can simplify the greatest common divisor:
\[ \gcd(b - a, 12p) = \gcd(b - a, 4p). \]

Now we evaluate the difference modulo $4p$. Since $p \ge 5$, we have $\gcd(3, 4p) = 1$. Thus, there exists a unique constant $c_r \in \mathbb{Z}_{4p}$ satisfying $3c_r + r \equiv 2p \pmod{4p}$.
Define $t \equiv k - l - c_r \pmod{4p}$. Substituting this into the difference modulo $4p$ yields:
\[ b - a \equiv 3(t + c_r) + r = 3t + (3c_r + r) \equiv 3t + 2p \pmod{4p}. \]
Therefore, $\gcd(b - a, 12p) = \gcd(3t + 2p, 4p)$. 

To determine when this greatest common divisor belongs to $\{2, p\}$, we establish a strict equivalence by analyzing the parity of $t$:
\begin{itemize}
    \item If $t$ is odd, then $3t + 2p$ is odd. Its greatest common divisor with $4p$ must be an odd divisor of $p$, meaning $\gcd(3t + 2p, 4p) \in \{1, p\}$. It equals $p$ if and only if $p \mid (3t + 2p) \implies p \mid 3t \implies p \mid t$. Since $t$ is odd, this is strictly equivalent to $\gcd(t, 4p) = p$.
    \item If $t$ is even, let $t = 2m$. Then $3t + 2p = 2(3m + p)$. Since this value is even, its greatest common divisor with $4p$ is also even, meaning it can only be $2$ (it cannot be $p$, as $p$ is odd). The condition $\gcd(2(3m + p), 4p) = 2$ holds if and only if $2 \nmid (3m + p)$ and $p \nmid (3m + p)$. Since $p$ is odd, $3m + p$ is odd if and only if $m$ is even. Furthermore, $p \nmid (3m + p)$ if and only if $p \nmid 3m \implies p \nmid m$. The combined conditions that $m$ is even and $p \nmid m$ are strictly equivalent to $4 \mid t$ and $p \nmid t$, which means exactly $\gcd(t, 4p) = 4$.
\end{itemize}
Consequently, $a \sim b$ in $H_r$ if and only if $\gcd(t, 4p) \in \{4, p\}$. 

The mapping $\varphi_r(3l) = (0, l + c_r)$ and $\varphi_r(3k + r) = (1, k)$ forms a bijection to $V(K_2 \otimes \ICG_{4p}(4, p))$ that strictly preserves adjacency. 
Finally, identifying each vertex $x \in \mathbb{Z}_{4p}$ with $(x \bmod 4, x \bmod p) \in \mathbb{Z}_4 \times \mathbb{Z}_p$, the greatest common divisor is $4$ if the first coordinates match and second are distinct, and $p$ if the second coordinates match and first are adjacent in $C_4$. This naturally yields $\ICG_{4p}(4, p) \cong C_4 \square K_p$, completing the proof.
\end{proof}

\begin{lemma}\label{lem:alpha-Gp}
Let $p \ge 5$ be a prime. The independence number of the graph $G_p$ is $4p$. An independent set attains this maximum size if and only if it is entirely contained within a single principal class. Furthermore, any independent set that mixes vertices from multiple classes has a strictly smaller size: it is bounded by $3p$ if it intersects exactly two classes, and by $3p+1$ if it intersects all three classes. Both bounds are sharp.
\end{lemma}

\begin{proof}
Let $C_0, C_1$, and $C_2$ denote the three principal classes of $G_p$, each of size $4p$. Let $A$ be an independent set in $G_p$. If $A$ is entirely contained within a single class $C_i$, its size is trivially bounded by $|C_i| = 4p$, which is explicitly achieved when $A = C_i$. 

Assume $A$ intersects at least two different classes, and let $A_i = A \cap C_i$. To analyze the structure of $A$, we map its vertices to the projection space $\mathbb{Z}_4 \times \mathbb{Z}_p$. This space consists of four horizontal \emph{layers}, defined as $L_k = \{k\} \times \mathbb{Z}_p$ for $k \in \{0, 1, 2, 3\}$. Let $x_k, y_k, z_k$ denote the number of vertices $A$ selects in layer $L_k$ from classes $C_0, C_1, C_2$, respectively. 

The structural limits of $G_p$ impose two strict rules on these variables:
\begin{enumerate}
    \item \textbf{Intra-layer rule:} Each layer $L_k$ induces a $K_p$ clique. Therefore, an independent set can select at most one vertex per class per layer ($x_k, y_k, z_k \le 1$), unless it selects vertices exclusively from one class. We say a class \emph{dominates} layer $L_k$ if it contributes $\ge 2$ vertices. If a layer is dominated, the other two classes must have 0 vertices in it, bounding the layer capacity to $p$. If no class dominates, the layer is \emph{mixed}, bounding its capacity to $3$.
    \item \textbf{Inter-layer rule:} Distinct classes in adjacent layers must be fully disjoint. This mandates $x_k + y_{k \pm 1} \le p$, $x_k + z_{k \pm 1} \le p$, and $y_k + z_{k \pm 1} \le p$ (indices modulo 4).
\end{enumerate}

We classify the independent set $A$ based on the number of dominated layers it contains:

\textbf{Case 1: At most two dominated layers.} \\
If $A$ has $\le 2$ dominated layers, they contribute at most $2p$ vertices. The remaining layers are mixed, contributing at most $3$ vertices each, yielding the algebraic bound $|A| \le 2p + 6$. Since $p \ge 5$, we inherently have $2p + 6 \le 3p + 1 < 4p$. (Note that for $p=5$, where $2p+6 = 16$, inter-layer constraints restrict the actual maximum size to only $14$). These configurations cannot yield a maximum independent set.

\textbf{Case 2: Four dominated layers.} \\
By the inter-layer rule, adjacent layers cannot be dominated by different classes (e.g., if $x_0 = p$ and $y_1 = p$, then $x_0 + y_1 = 2p > p$, a strict contradiction). Thus, if all four layers are dominated, they must be dominated by the exact same class, meaning $A \subseteq C_i$. Since we assumed $A$ mixes classes, this configuration is impossible.

\textbf{Case 3: Exactly three dominated layers.} \\
This is the only configuration where large mixed independent sets exist. Let $C_0$ dominate exactly three layers (without loss of generality, $L_0, L_1, L_3$). The total contribution from $C_0$ in these layers is bounded by $x_0, x_1, x_3 \le p$. The remaining layer, $L_2$, must be a mixed layer containing the elements from the other classes.

Since $A$ intersects other classes, $L_2$ must contain vertices from $C_1$ ($y_2 \ge 1$) and possibly $C_2$ ($z_2 \ge 1$). Because $L_2$ is mixed, the intra-layer rule forces $x_2, y_2, z_2 \le 1$. 
The presence of $y_2$ or $z_2$ triggers the inter-layer rule, strictly restricting the adjacent pure layers $L_1$ and $L_3$:
$$ x_1 + \max(y_2, z_2) \le p \implies x_1 \le p - 1 $$
$$ x_3 + \max(y_2, z_2) \le p \implies x_3 \le p - 1 $$

We now evaluate the total size $|A| = x_0 + x_1 + x_3 + (x_2 + y_2 + z_2)$:
\begin{itemize}
    \item If $A$ intersects \textbf{exactly two classes} ($C_0$ and $C_1$), then $y_2 = 1$ and $z_2 = 0$. The mixed layer $L_2$ contains at most $x_2 + y_2 \le 2$ vertices. The total size is bounded by $|A| \le p + (p - 1) + (p - 1) + 2 = 3p$.
    
    \item If $A$ intersects \textbf{all three classes}, then $y_2 = 1$ and $z_2 = 1$. The mixed layer $L_2$ contains at most $x_2 + y_2 + z_2 \le 3$ vertices. The total size is bounded by $|A| \le p + (p - 1) + (p - 1) + 3 = 3p + 1$.
\end{itemize}

\textbf{Equality Conditions (Maximal Mixed Independent Sets):} \\
To achieve the strict maximal bounds for mixed sets ($|A| = 3p$ or $|A| = 3p+1$), the configuration must precisely satisfy all capacity limits. This occurs if and only if the principal class provides exactly $3p - 1$ vertices, and the remaining classes provide exactly $1$ vertex each. Specifically: layer $L_0$ is fully saturated by $C_0$ ($X_0 = \mathbb{Z}_p$); the mixed layer $L_2$ selects exactly one vertex from each present class, all sharing the identical coordinate $v \in \mathbb{Z}_p$ ($X_2 = Y_2 = Z_2 = \{v\}$); and the adjacent layers $L_1$ and $L_3$ contain all $C_0$ vertices except the single coordinate adjacent to $v$ ($X_1 = X_3 = \mathbb{Z}_p \setminus \{v\}$). This geometric arrangement and its forced inter-layer exclusions are clearly illustrated in Figure \ref{fig:extremal-bound}.

Since $3p + 1 < 4p$ strictly holds for $p \ge 5$, mixed independent sets can never exceed the size of a single pure class. Therefore, $\alpha(G_p) = 4p$, strictly and uniquely achieved by the pure classes.
\end{proof}

\begin{figure}[H]
\centering
\begin{tikzpicture}[
    layer/.style={ellipse, draw, thick, minimum width=3.4cm, minimum height=2cm, align=center},
    dominant/.style={ellipse, fill=gray!25, minimum width=2.6cm, minimum height=1.4cm},
    subdominant/.style={ellipse, fill=gray!25, minimum width=2.1cm, minimum height=1.4cm},
    single/.style={circle, fill=black, inner sep=1.5pt},
    macroedge/.style={line width=1.2pt, double distance=1.5pt}
]

\node[layer] (L0) at (0, 4.5) {};
\node[layer] (L1) at (-4, 2.25) {};
\node[layer] (L3) at (4, 2.25) {};
\node[layer] (L2) at (0, 0) {};

\node[above] at (L0.north) {$L_0 \cong K_p$};
\node[above] at (L1.north) {$L_1 \cong K_p$};
\node[above] at (L3.north) {$L_3 \cong K_p$};
\node[below] at (L2.south) {$L_2 \cong K_p$ (Mixed layer)};

\draw[macroedge] (L0) -- (L1);
\draw[macroedge] (L0) -- (L3);
\draw[macroedge] (L1) -- (L2);
\draw[macroedge] (L3) -- (L2);

\node[dominant] (X0) at (L0) {};
\node at (X0) {$X_0$ ($x_0 = p$)};

\node[subdominant, anchor=east] (X1) at ([xshift=-1mm]L1.east) {};
\node at (X1) {$X_1$ ($x_1 = p-1$)};
\node[single, label={[font=\footnotesize]above:{$v \notin X_1$}}, draw=black, fill=white] (hole1) at ([xshift=6mm]L1.west) {};

\node[subdominant, anchor=west] (X3) at ([xshift=1mm]L3.west) {};
\node at (X3) {$X_3$ ($x_3 = p-1$)};
\node[single, label={[font=\footnotesize]above:{$v \notin X_3$}}, draw=black, fill=white] (hole3) at ([xshift=-6mm]L3.east) {};

\node[single, label=below:{$X_2, Y_2, Z_2 = \{v\}$}] (v2) at (L2) {};

\draw[->, thick, dashed] (v2) to[bend left=25] node[midway, left=2mm, font=\footnotesize, align=right] {Forces\\ exclusion} (hole1);
\draw[->, thick, dashed] (v2) to[bend right=25] node[midway, right=2mm, font=\footnotesize, align=left] {Forces\\ exclusion} (hole3);

\end{tikzpicture}
\caption{Visualization of the extremal independent set configuration yielding $|A| = 3p+1$ (Lemma \ref{lem:alpha-Gp}). The projection of elements from $C_1$ and $C_2$ onto a single node $v$ in $L_2$ strictly forces the exclusion of $v$ from the sets $X_1$ and $X_3$ in adjacent layers, capping their capacity at $p-1$.}
\label{fig:extremal-bound}
\end{figure}
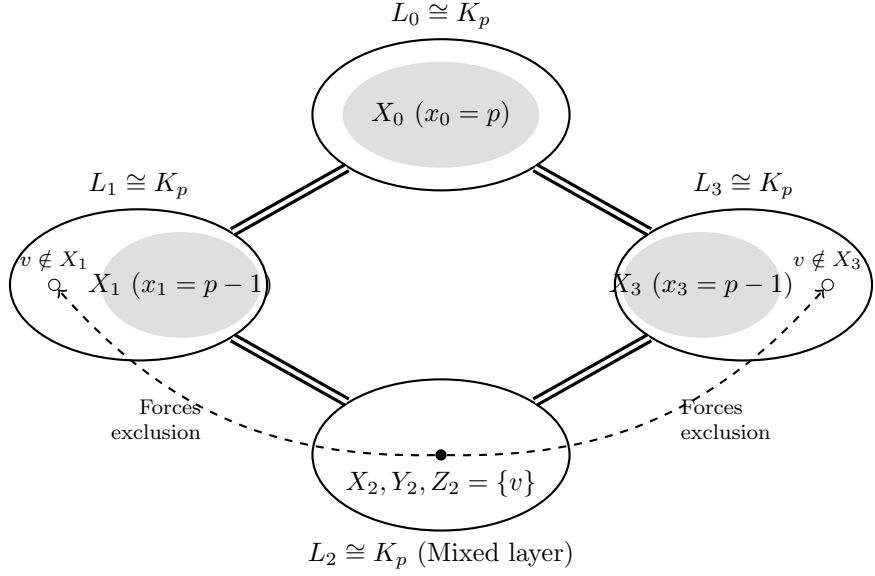

\begin{theorem}\label{thm:infinite-counterexamples}
For every prime $p \ge 5$, the graph $G_p = \ICG_{12p}(\{2, p\})$ is a uniquely $3$-colorable circulant graph with clique number $\omega(G_p) = 2$. Consequently, there exists an infinite family of uniquely colorable circulant graphs where $\chi(G) \neq \omega(G)$, providing a negative answer to Problem 3.6 in \cite{Klotz-2017}.
\end{theorem}

\begin{proof}
Any proper $3$-coloring partitions $V(G_p)$ into three independent sets. Since $|V(G_p)| = 12p$ and $\alpha(G_p) = 4p$ (Lemma~\ref{lem:alpha-Gp}), each color class must contain exactly $4p$ vertices. By the strict uniqueness condition in Lemma~\ref{lem:alpha-Gp}, the only independent sets of size $4p$ are the principal classes $C_0, C_1$, and $C_2$. Thus, the coloring is unique up to permutation of colors.

To determine the clique number of $G_p$, we apply the general structural theorem for integral circulant graphs with exactly two divisors, established by Ba\v{s}i\'{c} and Ili\'{c} (Theorem 3.1 in \cite{Basic-2009}). 
Applying their result to our graph $\ICG_{12p}(\{2, p\})$, we order the divisors such that $d_1 > d_2$, giving $d_1 = p$ and $d_2 = 2$. 
Since $d_2 = 2 > 1$ and $d_2$ does not divide $d_1$ (as $p \ge 5$ is an odd prime), we evaluate the clique number using the third case of their theorem:
\begin{equation*}
    \omega(G_p) = \max \left( f \left( \frac{12p}{p} \right), f \left( \frac{12p}{2} \right) \right) = \max \left( f(12), f(6p) \right),
\end{equation*}
where $f(x)$ denotes the smallest prime divisor of $x$. 
The smallest prime divisor of $12$ is $2$, and since $p \ge 5$, the smallest prime divisor of $6p$ is also $2$. Thus, $\omega(G_p) = \max(2, 2) = 2$.

Since $\chi(G_p) = 3 \neq 2 = \omega(G_p)$ strictly holds for every prime $p \ge 5$, the existence of this infinite family completely resolves the problem.
\end{proof}

\section{Uniquely colorable Cayley graphs over nonabelian groups}

In \cite{Klotz-2017}, Klotz and Sander posed an open question (Problem 2.4) asking whether there exists a uniquely colorable Cayley graph over a nonabelian group such that different color classes are left cosets of different subgroups. The algebraic motivation for this problem stems directly from Theorem 2.2(3) in their work, which establishes that the subgroups underlying any two color classes must be conjugate. Consequently, in an abelian group, all color classes are trivially forced to be cosets of the identical subgroup. To answer the problem in the affirmative, one must construct a Cayley graph over a nonabelian group where the color class containing the identity element forms a non-normal subgroup.

Rather than providing an isolated counterexample, we establish a general theoretical construction that yields an infinite family of such graphs, thereby completely resolving Problem 2.4 for any nonabelian group containing a non-normal subgroup. 

\begin{theorem}
Let $\Gamma$ be any nonabelian group of order $n$ and $U$ be a non-normal subgroup of $\Gamma$. The undirected Cayley graph $G = \text{Cay}(\Gamma, \Gamma \setminus U)$ is uniquely colorable, and its optimal color classes are left cosets of strictly distinct subgroups. 
\end{theorem}
\begin{proof}
Let $S = \Gamma \setminus U$. Since $U$ is a subgroup, $e \in U$, which guarantees $e \notin S$. Furthermore, $U$ is closed under inversion, meaning $x \in U \iff x^{-1} \in U$. Consequently, $x \in S \iff x^{-1} \in S$, ensuring that $S = S^{-1}$. Thus, $G$ is a well-defined undirected Cayley graph without loops.

By Lagrange's theorem, the order of $U$ strictly divides $n$. Let $k = n/|U|$. We partition the vertex set $\Gamma$ into exactly $k$ distinct right cosets $C_i = U d_i$ for $1 \le i \le k$. Here, the set $\{d_1, d_2, \dots, d_k\}$ represents a complete right transversal of $U$ in $\Gamma$, where we conventionally set $d_1 = e$. 

By the definition of the Cayley graph, two vertices $x, y \in \Gamma$ are non-adjacent if and only if $xy^{-1} \notin S$, which is strictly equivalent to $xy^{-1} \in U$. 
We evaluate adjacency in two structural cases:
\begin{itemize}
    \item \textbf{Intra-coset adjacency:} Let $x, y$ belong to the same right coset $C_i$. We can express them as $x = u_1 d_i$ and $y = u_2 d_i$ for some $u_1, u_2 \in U$. Their evaluation yields $xy^{-1} = (u_1 d_i)(u_2 d_i)^{-1} = u_1 d_i d_i^{-1} u_2^{-1} = u_1 u_2^{-1}$. Since $U$ is a subgroup, $u_1 u_2^{-1} \in U$. Therefore, there are no edges within any right coset, making each $C_i$ an independent set of size $n/k$.
    \item \textbf{Inter-coset adjacency:} Let $x \in C_i$ and $y \in C_j$ for $i \neq j$. Thus, $x = u_1 d_i$ and $y = u_2 d_j$. Suppose, for the sake of contradiction, that they are non-adjacent, meaning $xy^{-1} \in U$. This implies $u_1 d_i d_j^{-1} u_2^{-1} = u_3 \in U$. Rearranging this gives $d_i d_j^{-1} = u_1^{-1} u_3 u_2 \in U$, which fundamentally implies $U d_i = U d_j$. This contradicts the construction that $C_i$ and $C_j$ are distinct cosets. Therefore, $xy^{-1} \notin U$, meaning $xy^{-1} \in S$. Consequently, every vertex in $C_i$ is connected to every vertex in $C_j$.
\end{itemize}

This evaluation proves that $G$ is isomorphic to the complete multipartite graph $K_{n/k, \dots, n/k}$ containing exactly $k$ parts. Complete multipartite graphs are rigidly uniquely colorable \cite{Klotz-2017}, guaranteeing that the right cosets $C_1, \dots, C_k$ constitute the only possible optimal color classes of $G$.

To satisfy the specific topological conditions of Problem 2.4, each color class must be expressed as a \textit{left} coset of some generating subgroup $U(C_i)$. Utilizing basic group theory, any right coset can be rewritten as a left coset of a conjugate subgroup:
$$ C_i = U d_i = d_i (d_i^{-1} U d_i) = d_i U_{d_i}, $$
where $U_{d_i} = d_i^{-1} U d_i$ is guaranteed to be a valid subgroup. Thus, the color class $C_i$ is explicitly a left coset of the subgroup $U(C_i) = d_i^{-1} U d_i$. 

Finally, since $U$ is explicitly chosen to be a non-normal subgroup, there must exist some arbitrary element $g \in \Gamma$ such that $g^{-1}Ug \neq U$. By the established partition, this element can be written as $g = u d_k$ for some $u \in U$ and some transversal representative $d_k$. Evaluating the conjugation yields:
$$ g^{-1} U g = (u d_k)^{-1} U (u d_k) = d_k^{-1} u^{-1} U u d_k. $$
Because $u \in U$, the inner conjugation inherently satisfies $u^{-1} U u = U$. The expression therefore simplifies exactly to $d_k^{-1} U d_k$. Since $g^{-1}Ug \neq U$, it rigorously follows that $d_k^{-1} U d_k \neq U$. 

Consequently, at least two color classes (specifically the identity class $C_1 = U$ and the class $C_k = d_k U_{d_k}$) are left cosets of strictly distinct subgroups, fully resolving the problem.
\end{proof}

The smallest group exhibiting this general topological property is the symmetric group $S_3$. Selecting the non-normal subgroup $U = \{e, (12)\}$ immediately generates the unique complete tripartite graph $K_{2,2,2}$, where the respective left-coset subgroups $U(C_i)$ explicitly evaluate to the three distinct subgroups $\{e, (12)\}$, $\{e, (13)\}$, and $\{e, (23)\}$.

\section{Conclusion and further research}

In this paper, we resolved two open problems regarding uniquely colorable Cayley graphs, demonstrating that the structural constraints imposed by unique colorability are significantly less restrictive than previously hypothesized. Had the original conjectures held true, the class of uniquely colorable Cayley graphs would have been rigidly confined to extreme topological ($\chi = \omega$) and algebraic (conjugate generating subgroups) equivalences. Instead, we established that unique colorability in circulant graphs does not mandate the equality of chromatic and clique numbers, and we proved that nonabelian Cayley graphs can support uniquely generated color classes deriving from strictly distinct, non-normal subgroups. In both cases, the resolutions required moving beyond standard spectral bounds or traditional left-coset analytical methods. We relied on exact structural isomorphisms to compute the true independence number and utilized right-coset evaluations to bypass conjugacy limitations.

While this paper focused exclusively on the rigid constraints of uniquely colorable graphs, a natural progression is to investigate graphs where the chromatic and clique numbers perfectly align not just globally, but across all induced subgraphs. In \cite{klotz07}, Klotz and Sander successfully characterized all perfect unitary Cayley graphs. Since unitary Cayley graphs represent only a specific subset of integral circulant graphs, determining the full classification of perfect integral circulant graphs remains a compelling open problem. Consequently, utilizing the exact combinatorial methods and topological projections developed in this paper to establish structural bounds for perfectness within the broader class of integral circulant graphs presents a highly promising direction for future research.

\section*{Acknowledgment}

This research was supported by the research project of the Ministry of Education, Science and Technological Development of the Republic of Serbia (No.~451-03-34/2026-03/200124).



\footnotesize

 \begin{thebibliography}{00}















\bibitem{ahmadi03} A. Ahmadi, R. Belk, C. Tamon, C. Wendler, On mixing of continuous-time quantum walks on some circulant graphs, {\it Quantum Inf. Comput.} {\bf 3} (2003) 611--618.

\bibitem{Ba14} M. Ba\v{s}i\'{c}, Which weighted circulant networks have perfect state transfer?, {\it Inform. Sci.} {\bf 257} (2014) 193--209.

\bibitem{Basic-2024} M. Ba\v{s}i\'{c}, Minimal spread of integral circulant graphs, {\it Discret. Appl. Math.} {\bf 342} (2024) 317--333.

\bibitem{Basic-2025} M. Ba\v{s}i\'{c}, The least eigenvalues of integral circulant graphs, {\it J. Algebraic Combin.} {\bf 61} (2025).

\bibitem{Basic-2009} M. Ba\v{s}i\'{c}, A. Ili\'{c}, On the clique number of integral circulant graphs, {\it Appl. Math. Lett.} {\bf 22} (2009) 1406--1411.

\bibitem{BaIlSt23} M. Ba\v{s}i\'{c}, A. Ili\'{c}, A. Stamenkovi\'{c}, Maximal diameter of integral circulant graphs, {\it Inform. Comput.} {\bf 298} (2024) 105208.

\bibitem{Basic-LAA-2025} M. Ba\v{s}i\'{c}, Graphs with three and four distinct eigenvalues based on circulants, {\it Linear Algebra Appl.} {\bf 718} (2025) 30--57.

\bibitem{Bhakta-2025} K. Bhakta, B. Bhattacharjya, Periodicity and perfect state transfer of Grover walks on quadratic unitary Cayley graphs, {\it Quantum Inf. Process.} {\bf 24} (2025) 260.

\bibitem{Bhakta-2026} K. Bhakta, B. Bhattacharjya, State transfer in Grover walks on unitary and quadratic unitary Cayley graphs over finite commutative rings, {\it Discrete Math.} {\bf 349} (2026) 115151.

\bibitem{Davis70} P. J. Davis, {\it Circulant Matrices}, Wiley, New York, 1970.

\bibitem{Godsil17} G. Godsil, H. Zhan, Uniform mixing on Cayley graphs, {\it Electron. J. Comb.} {\bf 24} (2017) \#P3.20.

\bibitem{HardyWright} G. H. Hardy, E. M. Wright, {\it An introduction to the Theory of Numbers}, 5th ed., Clarendon Press, Oxford University Press, New York, 1979.

\bibitem{Hwang03} F. K. Hwang, A survey on multi-loop networks, {\it Theor. Comput. Sci.} {\bf 299} (2003) 107--121.

\bibitem{Basic-2010} A. Ili\'{c}, M. Ba\v{s}i\'{c}, On the chromatic number of integral circulant graphs, {\it Comput. Math. Appl.} {\bf 60} (2010) 144--150.

\bibitem{klotz07} W. Klotz, T. Sander, Some properties of unitary Cayley graphs, {\it Electron. J. Comb.} {\bf 14} (2007) \#R45.

\bibitem{Klotz-2017} W. Klotz, T. Sander, Uniquely colorable Cayley graphs, {\it Ars Math. Contemp.} {\bf 12} (2017) 155--165.

\bibitem{Kubota-2025} S. Kubota, K. Yoshino, Circulant graphs with valency up to 4 that admit perfect state transfer in Grover walks, {\it J. Combin. Theory Ser. A} {\bf 216} (2025) 106064.

\bibitem{Moenius-2023} K. M\"{o}nius, W. So, How many non-isospectral integral circulant graphs are there?, {\it Australas. J. Combin.} {\bf 86} (2023) 320--335.

\bibitem{Sander2021} J. W. Sander, The geometric kernel of integral circulant graphs, {\it Electron. J. Comb.} {\bf 28} (2021) \#P3.33.

\bibitem{SanderSander2023} J. W. Sander, T. Sander, Characterisation of all integral circulant graphs with multiplicative divisor sets and few eigenvalues, {\it J. Algebr. Comb.} {\bf 58} (2023) 993--1017.

\bibitem{Saxena-2007} N. Saxena, S. Severini, I. Shparlinski, Parameters of integral circulant graphs and periodic quantum dynamics, {\it Int. J. Quantum Inf.} {\bf 5} (2007) 417--430.

\bibitem{wasin} W. So, Integral circulant graphs, {\it Discrete Math.} {\bf 306} (2006) 153--158.






\end{thebibliography}
\end{document}